\documentclass[12pt,a4paper]{article}
\usepackage{amssymb,graphics,amsmath,amsthm}
\title{}
\date{}
\newtheorem{theorem}{Theorem}
\newtheorem{lemma}{Lemma}
\newtheorem{corollary}{Corollary}

\newtheorem{definition}{Def{\kern0pt}inition}
\newtheorem{remark}{Remark}
\newtheorem{example}{Example}
\newcommand{\QED}{\hfill \rule{.1in}{.1in}}
\begin{document}
\begin{center}
\Large{Ef{\kern0pt}f{\kern0pt}iciency analysis of simple perturbed pairwise comparison matrices}\\[1cm]
\end{center}
\begin{center}
\normalsize
Krist\'of \'Abele-Nagy \\
\small{
Laboratory on Engineering and Management Intelligence, \\
Research Group of Operations Research and Decision Systems, \\
Institute for Computer Science and Control, \\
Hungarian Academy of Sciences (MTA SZTAKI); \\
Department of Operations Research and Actuarial Sciences,\\
Corvinus University of Budapest, Hungary \\
\verb|abele-nagy.kristof@sztaki.mta.hu|  } \\[5mm]

\normalsize
S\'andor Boz\'oki$^{1}$ \\
\small{
Laboratory on Engineering and Management Intelligence, \\
Research Group of Operations Research and Decision Systems, \\
Institute for Computer Science and Control, \\
Hungarian Academy of Sciences (MTA SZTAKI); \\
Department of Operations Research and Actuarial Sciences,\\
Corvinus University of Budapest, Hungary \\
\verb|bozoki.sandor@sztaki.mta.hu|  \\
\verb|http://www.sztaki.mta.hu/%7Ebozoki| }
\end{center}

\footnotetext[1]{corresponding author}

\begin{abstract}
Ef{\kern0pt}f{\kern0pt}iciency, the basic concept of multi-objective optimization
is investigated for the class of pairwise comparison matrices.
A weight vector is called ef{\kern0pt}f{\kern0pt}icient if
no alternative weight vector exists such that every pairwise ratio of the latter's components
is at least as close to the corresponding element of the pairwise comparison matrix
as the one of the former's components is, and the
latter's approximation is strictly better in at least one position.
A pairwise comparison matrix is called simple perturbed if it dif{\kern0pt}fers
from a consistent pairwise comparison matrix in one element and its reciprocal.
One of the classical weighting methods, the eigenvector method is analyzed.
It is shown in the paper that the principal right eigenvector of a simple perturbed
pairwise comparison matrix is ef{\kern0pt}f{\kern0pt}icient.
An open problem is exposed: the search for a necessary and suf{\kern0pt}f{\kern0pt}icient condition of that
the principal right eigenvector is ef{\kern0pt}f{\kern0pt}icient.\\

\end{abstract}

\section{Introduction}
To determine the importance and/or the weights of criteria as well as comparing
the alternatives in multi-attribute decision making problems are of crucial importance.

A ratio scale matrix $\mathbf{A}=[a_{ij}]$ is a positive square matrix with the reciprocal property
$a_{ij} = 1/a_{ji},$ where $i,j=1,2,\ldots,n,$ and $a_{ii} =1, i=1,2,\ldots,n.$
An $a_{ij}$ entry from $\mathbb{R}$ in $\mathbf{A}$ represents the strength or
the relative importance ratio of decision alternative $i$ over alternative $j$ with respect to a
common criterion. The relative importance ratios are usually elicited from people who
produce $n(n-1)/2$ subjective judgments on each possible pair of the alternatives. Once such a
matrix called pairwise comparison matrix (PCM) has been constructed ($n\geq3$), the objective is
to extract the implicit (positive) weights of the various alternatives.

It is often easier to make these comparisons in pairs, by answering the questions
\begin{itemize}
    \item How many times is a criterion more important than another criterion?
    \item With respect to a given criterion, how many times is an alternative better than another alternative?
    \item How many times is a voting power of a decision maker is greater than that of another decision maker?
    \item How many times is a scenario more probable than another one?
\end{itemize}
Numerical answers can be arranged in a matrix.
Pairwise comparison matrix $\mathbf{A}=[a_{ij}]_{i,j=1,\dots,n}$ thus has the following properties:
$a_{ij}>0$, $a_{ji}=1/a_{ij}$, $i,j=1,2,\dots,n$.
Let $\mathcal{PCM}_n$ denote the class of pairwise comparison matrices of size $n \times n,$
where $n \geq 3.$
$\mathbf{A}$ is called consistent, if $a_{ik}a_{kj}=a_{ij}$ holds for all $i,j,k=1,2,\ldots,n$,
otherwise it is called inconsistent.

Once the decision makers have provided their assessments on the pairwise ratios,
the objective is to f{\kern0pt}ind a weight vector
$\mathbf{w}=(w_1,w_2,\dots,w_n)^{\top}$
such that these ratios, $\frac{w_i}{w_j}$ be found that are as close as possible
to $a_{ij}$s for all $i,j=1,2,\dots,n$. The column vector
 $\mathbf{w}$ is usually normalized so that $\sum_i w_i =1$.
A weight vector can be extracted from a PCM in several ways
\cite{BajwaChooWedley2008,ChooWedley2004,Dijkstra2013,GolanyKress1993},
however, the paper investigates the eigenvector method \cite{Saaty1977} only.

If $\mathbf{A}$ is consistent, then the eigenvector equation $\mathbf{A}\mathbf{w}=n\mathbf{w}$ holds,
and $w_i/w_j=a_{ij}, \linebreak i,j=1,2,\dots,n, \, w_i>0, \sum_i w_i=1$ \cite{Saaty1977}.
This idea is extended to the general case as follows.
The eigenvector method, proposed by Saaty \cite{Saaty1977},
obtains the weight vector from the eigenvector equation $\mathbf{A}\mathbf{w}^{EM}=
\lambda_{\max}\mathbf{w}^{EM}$,
where $\lambda_{\max}$ denotes the Perron-eigenvalue (also known as
maximal or principal eigenvalue) of $\mathbf{A}$,
and $\mathbf{w}^{EM}$ is the corresponding principal right eigenvector.
\begin{equation}
\lambda_{\max} \geq n   \label{EQ:lambdamaxn}  
\end{equation}
and equality holds if and only if $\mathbf{A}$ is consistent.
$\mathbf{w}^{EM}$ is positive and unique up to a scalar multiplication.
Note that $\lambda_{\max}$ shall also be denoted by
$\lambda_{\delta}$ in the beginning of Section 2 in order
to emphasize its dependence on parameter $\delta$.
Later on $\lambda_{\delta}$ shall be shortened by $\lambda$.
Once $\mathbf{w}^{EM}$ is  computed, the  ratio $\frac{w^{EM}_i}{w^{EM}_j}$
can be compared to $a_{ij}$. In the consistent case, as mentioned before,
$\frac{w^{EM}_i}{w^{EM}_j} = a_{ij}$ for all $i,j=1,2,\dots,n$.

We now introduce the term ef{\kern0pt}f{\kern0pt}iciency also known as Pareto optimality,
or nondominatedness \cite[Chapter 2.4]{Zeleny1982} which is considered to be a
basic concept of multi-objective optimization and referring to the weight vectors derived from PCMs.
Let $\mathbf{A} =
\left[
a_{ij}
\right]_{i,j=1,\ldots,n} \in \mathcal{PCM}_n$  and
$\mathbf{w} = (w_1, w_2, \ldots, w_n)^{\top}$ be a positive weight vector.
\begin{definition} \label{def:DefinitionEfficient}  
 A positive weight vector  $\mathbf{w}$ is called \emph{ef{\kern0pt}f{\kern0pt}icient}
if no other positive weight vector
$\mathbf{w^{\prime}} = (w^{\prime}_1, w^{\prime}_2, \ldots, w^{\prime}_n)^{\top}$
exists such that
\begin{align}
 \left|a_{ij} - \frac{w^{\prime}_i}{w^{\prime}_j} \right| &\leq \left|a_{ij} - \frac{w_i}{w_j} \right| \qquad \text{ for all } 1 \leq i,j \leq n,  \\
 \left|a_{k{\ell}} - \frac{w^{\prime}_k}{w^{\prime}_{\ell}} \right| &<  \left|a_{k{\ell}} - \frac{w_k}{w_{\ell}} \right|  \qquad \text{ for some } 1 \leq k,\ell \leq n.
\end{align}
\end{definition}
 A weight vector
$\mathbf{w}$ is called \emph{inef{\kern0pt}f{\kern0pt}icient} if it is not ef{\kern0pt}f{\kern0pt}icient. \\

It follows from the def{\kern0pt}inition, that $\mathbf{w}^{EM}$ is ef{\kern0pt}f{\kern0pt}icient for every consistent
PCM as $a_{ij} = \frac{w^{EM}_i}{w^{EM}_j}$ for all $i,j=1,\ldots,n.$\\

Blanquero et al. \cite{BlanqueroCarrizosaConde2006} investigated
several necessary and suf{\kern0pt}f{\kern0pt}icient conditions on
ef{\kern0pt}f{\kern0pt}iciency.
One of these ef{\kern0pt}f{\kern0pt}iciency conditions is utilized in our
paper, which applies a directed graph representation.

\begin{definition} \label{def:DirectedGraphEfficient}  
Let $\mathbf{A} =
\left[
a_{ij}
\right]_{i,j=1,\ldots,n} \in \mathcal{PCM}_n$  and
$\mathbf{w} = (w_1, w_2, \ldots, w_n)^{\top}$ be a positive weight vector.
 A directed graph $G:=(V,\overrightarrow{E})_{\mathbf{A},\mathbf{w}}$ is def{\kern0pt}ined as follows:
$V=\{1,2,\ldots,n\} $ and
\[
 \overrightarrow{E} = \left\{ \text{arc(}  i \rightarrow j\text{)}   \left|  \frac{w_i}{w_j} \geq a_{ij}, i \neq j \right. \right\}.
\]
\end{definition}
\begin{theorem} \cite[Corollary 10]{BlanqueroCarrizosaConde2006} \label{thm:TheoremDirectedGraphEfficient}
Let $\mathbf{A} \in \mathcal{PCM}_n$.
A weight vector $\mathbf{w}$ is ef{\kern0pt}f{\kern0pt}icient if and only if
$G=(V,\overrightarrow{E})_{\mathbf{A},\mathbf{w}}$
is  a strongly connected digraph, that is,
there exist directed paths from $i$ to $j$ and from $j$ to $i$ for all
pairs of  nodes  $i , j$.
\end{theorem}

Blanquero et al. \cite[Section 3]{BlanqueroCarrizosaConde2006} also showed that the
principal right eigenvector can be inef{\kern0pt}f{\kern0pt}icient.
 This remarkable result was recalled by
Bajwa, Choo and Wedley \cite{BajwaChooWedley2008}
 and by  Conde and P\'erez \cite{CondePerez2010}
 and also by Fedrizzi \cite{Fedrizzi2013}.

Another numerical example is provided here to illustrate (in)ef{\kern0pt}f{\kern0pt}iciency and its digraph representation.
\begin{example} \label{example1}  
Let $ \mathbf{A} \in \mathcal{PCM}_4$ as follows:
\[
\mathbf{A} =
\begin{pmatrix}
$\,$  1   $\,\,$ & $\,\,$   1   $\,\,$ & $\,\,$  1/5 $\,\,$ & $\,\,$     1/5  $\,$    \\
$\,$  1   $\,\,$ & $\,\,$   1   $\,\,$ & $\,\,$  1/3 $\,\,$ & $\,\,$     1/7  $\,$    \\
$\,$  5  $\,\,$ & $\,\,$   3   $\,\,$ & $\,\,$   1  $\,\,$ & $\,\,$     1/4  $\,$   \\
$\,$  5  $\,\,$ & $\,\,$   7   $\,\,$ & $\,\,$   4  $\,\,$ & $\,\,$      1   $\,$
\end{pmatrix}.
\]
The principal right eigenvector of $\mathbf{A}$ and the consistent
approximation of $\mathbf{A}$, generated by $\mathbf{w}^{EM}$, are displayed, truncated at
8 and 4 correct digits, respectively:
\[
\mathbf{w}^{EM} =
\begin{pmatrix}
  0.07777933    \\
  0.07732534    \\
  0.24353753     \\
  0.60135778
\end{pmatrix},
\quad
\left[ \frac{w^{EM}_i}{w^{EM}_j} \right] =
\begin{pmatrix}
$\,$       1  $\,\,$ & $\,\,$ 1.0058 $\,\,$ & $\,\,$ 0.3193 $\,\,$ & $\,\,$ 0.1293 $\,$    \\
$\,$ 0.9941 $\,\,$ & $\,\,$      1   $\,\,$ & $\,\,$ 0.3175 $\,\,$ & $\,\,$ 0.1285 $\,$    \\
$\,$ 3.1311 $\,\,$ & $\,\,$ 3.1495 $\,\,$ & $\,\,$       1  $\,\,$ & $\,\,$ 0.4049 $\,$   \\
$\,$ 7.7315 $\,\,$ & $\,\,$ 7.7769 $\,\,$ & $\,\,$ 2.4692 $\,\,$ & $\,\,$       1  $\,$
\end{pmatrix}.
\]


Let us apply Def{\kern0pt}inition \ref{def:DirectedGraphEfficient} in order to draw the
digraph associated to matrix $ \mathbf{A}$ and its principal right eigenvector $\mathbf{w}^{EM}$.
The digraph in Figure 1 cannot be strongly connected because no arc leaves node 2.
\end{example}  

\unitlength 1mm
\begin{center}
\begin{picture}(100,40)
\put(39,10){\resizebox{20mm}{!}{\rotatebox{0}{\includegraphics{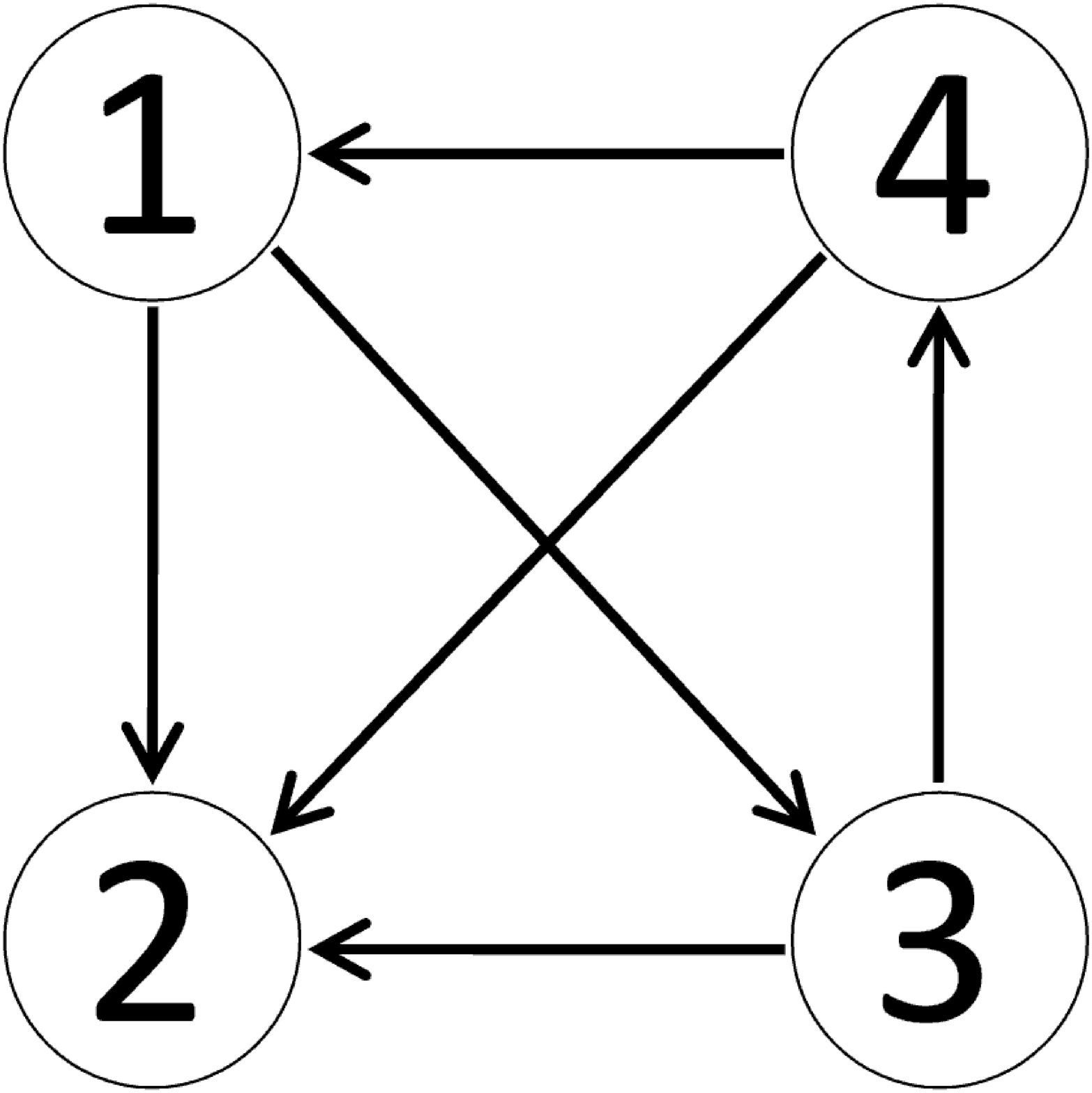}}}}
\put(-10,3){\makebox{\textbf{Figure 1.}
The principal right eigenvector in Example 1.1 is inef{\kern0pt}f{\kern0pt}icient, }}
\put(7,-1){\makebox{because its associated digraph is not strongly connected }}
\end{picture}
\end{center}

\bigskip
It might be instructive to see another, direct evidence why $\mathbf{w}^{EM}$ is inef{\kern0pt}f{\kern0pt}icient.
Let us increase the second coordinate of $\mathbf{w}^{EM}$ until it reaches ${w}^{EM}_1$, i.e.,
def{\kern0pt}ine $ \mathbf{w}^{\prime} := ({w}^{EM}_1, {w}^{EM}_1, {w}^{EM}_3, {w}^{EM}_4)^{\top}$. Then
\[
\mathbf{w}^{\prime} :=
\begin{pmatrix}
  0.07777933    \\
  0.07777933    \\
  0.24353753     \\
  0.60135778
\end{pmatrix},
\quad
\left[ \frac{w^{\prime}_i}{w^{\prime}_j} \right] =
\begin{pmatrix}
$\,$       1  $\,\,$ & $\,\,$      \textbf{1}   $\,\,$ & $\,\,$ 0.3193 $\,\,$ & $\,\,$ 0.1293 $\,$    \\
$\,$ \textbf{1} $\,\,$ & $\,\,$      1   $\,\,$ & $\,\,$ \textbf{0.3193} $\,\,$ & $\,\,$ \textbf{0.1293} $\,$    \\
$\,$ 3.1311 $\,\,$ & $\,\,$ \textbf{3.1311} $\,\,$ & $\,\,$       1  $\,\,$ & $\,\,$ 0.4049 $\,$   \\
$\,$ 7.7315 $\,\,$ & $\,\,$ \textbf{7.7315} $\,\,$ & $\,\,$ 2.4692 $\,\,$ & $\,\,$       1  $\,$
\end{pmatrix}.
\]
It can be seen that (with $\mathbf{w} = \mathbf{w}^{EM})$
the strict inequality (3) in Def{\kern0pt}inition \ref{def:DefinitionEfficient}
holds exactly for the non-diagonal elements of the second row/column,
marked by bold. For all other entries inequality (2) holds with equality. \\

Example 1.1 above illustrates that Theorem \ref{thm:TheoremDirectedGraphEfficient}
is powerful and easy to apply.\\

Boz\'oki \cite{Bozoki2014} showed that the principal right eigenvector of
the parametric pairwise comparison matrix
\[
\mathbf{A}(p,q) =
\begin{pmatrix}
   1     &     p    &     p    &   p     &  \ldots &    p    &  p   \\
 1/p     &     1    &     q    &   1     &  \ldots &    1    &  1/q     \\
 1/p     &    1/q   &     1    &   q     &  \ldots &    1    &  1       \\
  \vdots &   \vdots &   \vdots &  \ddots &         &  \vdots &   \vdots  \\
  \vdots &   \vdots &   \vdots &         &  \ddots &  \vdots &   \vdots  \\
 1/p     &     1    &     1    &   1     &  \ldots &    1    &   q       \\
 1/p     &     q    &     1    &   1     &  \ldots &    1/q    &   1
\end{pmatrix}\in \mathcal{PCM}_n,
\]
where $n \geq 4,$ $p > 0$ and $1 \neq q > 0,$ is inef{\kern0pt}f{\kern0pt}icient. \\

However, the general problem is still open:\\[3mm]
\textbf{Question 1.1. }
What is the necessary and suf{\kern0pt}f{\kern0pt}icient condition of that
the principal right eigenvector is ef{\kern0pt}f{\kern0pt}icient? \\

Our signif{\kern0pt}icantly more moderate objective is to provide a new suf{\kern0pt}f{\kern0pt}icient condition.
The main contribution of the paper is the ef{\kern0pt}f{\kern0pt}iciency analysis of
another special class of PCMs.
Departing from a consistent PCM, let us modify a single element
and its reciprocal, which results in a simple perturbed PCM.
 It will be shown in this paper that
the principal right eigenvector of a simple perturbed
pairwise comparison matrix is ef{\kern0pt}f{\kern0pt}icient.

\section{Simple perturbed pairwise comparison matrix}
Let $ x_1 , x_2 , \dots , x_{n-1}$ be arbitrary positive numbers.
Let $\delta$ represent a
multiplicative perturbation factor with an arbitrary positive number, where $\delta \neq 1.$
A simple perturbed PCM is def{\kern0pt}ined as follows:
\begin{equation}
\label{EQ:SimplePerturbedPCM}  
 \mathbf{A}_{\delta}  =
\left(\begin{array}{ccccc}
1 & x_1 \delta & x_2 & \dots & x_{n-1} \\ \frac{1}{x_1 \delta} & 1 & \frac{x_2}{x_1} & \dots & \frac{x_{n-1}}{x_1} \\
\frac{1}{x_2} & \frac{x_1}{x_2} & 1 & \dots & \frac{x_{n-1}}{x_2} \\
\vdots & \vdots & \vdots & \ddots & \vdots \\
\frac{1}{x_{n-1}} & \frac{x_1}{x_{n-1}} & \frac{x_2}{x_{n-1}} & \dots & 1
\end{array} \right)\in \mathcal{PCM}_n.
\end{equation}
As it is apparent from matrix $\mathbf{A}_{\delta}$, such a simple perturbed PCM is
constructed from a consistent PCM by entering a perturbation factor $\delta$
at its entry $a_{12}$, while its reciprocal entry $a_{21}$ is multiplied by $\frac{1}{\delta}$.
We remark, that any location of the perturbed pair of elements in the PCM does not involve the loss of generality.

PCMs that can be made consistent by the modif{\kern0pt}ication of one/two/three
entries have been analyzed by Boz\'{o}ki, F\"{u}l\"{o}p and Poesz \cite{BozokiFulopPoesz2011}.
Out of 20 PCMs of size $4 \times 4$ originated from real decision problems,
6 PCMs were simple perturbed. See \cite[Table 1]{BozokiFulopPoesz2011} for more
f{\kern0pt}indings.
The more general idea of comparing two PCMs that dif{\kern0pt}fer from each other in a single entry
(and its reciprocal) has been applied by Cook and Kress \cite[Axiom 2]{CookKress1988} and
also by Brunelli and Fedrizzi \cite[Axiom 4]{BrunelliFedrizzi2015}.

 Let $  \lambda_{\delta} $ denote the Perron-eigenvalue of
$\mathbf{A}_{\delta}$.
It follows from (\ref{EQ:lambdamaxn}) that if $\delta \neq 1$, then $ \lambda_{\delta} > n$.
Formally $ \lambda_1 = n$ holds, but we shall not consider $\mathbf{A}_1$ a simple perturbed PCM,
as it is consistent.

Farkas \cite{Farkas2007} shows, that $\lambda_{\delta}$
can be obtained from the following equation:
\[ \lambda^3_{\delta}-n{ \lambda^2_{\delta} }-(n-2)\left(\delta+\frac{1}{\delta}-2\right)=0.\]

One can write an explicit formula for $\lambda_{\delta}$:
\begin{equation*}
 \lambda_{\delta}  =
\frac{1}{6}
\sqrt[3]{ \frac{ B+12\sqrt{3 C}}{\delta} }
+
\frac{2}{3}
\sqrt[3]{ \frac{\delta}{ B +12\sqrt{3 C }} }
+
\frac{1}{3}n,
\end{equation*}
where\\
$ B  = 8 n^3 \delta + 108 n \delta^2  - 216 n \delta + 108 n - 216 \delta^2 + 432 \delta - 216,$ \\
$ C  =
4 n^4 \delta^3
-8 n^4 \delta^2
+4 n^4 \delta
-8  n^3 \delta^3
+16 n^3 \delta^2
-8 n^3 \delta
+27 n^2 \delta^4
+162 n^2 \delta^2$ \\
\indent
$-108 n^2 \delta^3
-108 n^2 \delta
+ 27 n^2
-108 n \delta^4
+432 n \delta^3
-648 n \delta^2
+432 n \delta
-108 n$ \\
\indent
$
+108 \delta^4
-432 \delta^3
+648 \delta^2
-432 \delta
+ 108,$ \\
even if the main results of the paper can be proved without the expanded formula above.
In the remainder of the paper $\lambda$ denotes $ \lambda_{\delta}
 = \lambda_{\max}(\mathbf{A}_{\delta}) $.

Farkas, R\'ozsa and Stubnya \cite{FarkasRozsaStubnya1999} developed a general method
to write the explicit form of the principal right eigenvector, when the perturbed elements
are in the same row/column of the PCM.
According to Farkas \cite[Formula 26]{Farkas2007}, the principal right eigenvector
 of the simple perturbed PCM $ \mathbf{A}_{\delta}$
can be written as
\begin{equation} 
\label{EQ:Farkas26}  
\mathbf{w}^{EM}=
\begin{pmatrix}
{w}_1^{EM} \\
{w}_2^{EM} \\
\vdots \\
{w}_i^{EM} \\
\vdots
\end{pmatrix}
=
\begin{pmatrix}
\lambda-1+\delta \\ \frac{1}{x_1}(\lambda-1+\frac{1}{\delta}) \\ \vdots \\ \frac{1}{x_{i-1}}\lambda \frac{\lambda-2}{n-2}  \\ \vdots
\end{pmatrix}, \qquad
i=3,4,\dots,n.   
\end{equation}
The principal right eigenvector of PCM $ \mathbf{A}_{\delta}$
can be written in an alternative way (Farkas \cite[Formula 24]{Farkas2007}):
\begin{equation} 
\label{EQ:Farkas24}        
\mathbf{w}^{EM}=
\begin{pmatrix}
{w}_1^{EM} \\
{w}_2^{EM} \\
\vdots \\
{w}_i^{EM} \\
\vdots
\end{pmatrix}
=
c
\begin{pmatrix}
\lambda (\lambda-n+1) \\ \frac{1}{x_1}\left[\lambda-(1-\frac{1}{\delta})(\lambda-n+2)\right] \\ \vdots \\ \frac{1}{x_{i-1}}(\lambda-1+\frac{1}{\delta}) \\ \vdots
\end{pmatrix}, \qquad
i=3,4,\dots,n,
\end{equation}
where scalar $c$ can be expressed as $\frac{\lambda-1+\delta}{\lambda (\lambda-n+1)}$.
Both formulas (\ref{EQ:Farkas26})-(\ref{EQ:Farkas24}) shall be applied, depending on our purpose.

\begin{remark}
\label{remark:3x3} 
Every PCM of size $3 \times 3$ is either consistent or simple perturbed.
\end{remark}

\section{Ef{\kern0pt}f{\kern0pt}iciency of the principal right eigenvector of a simple perturbed pairwise comparison matrix}

Note that the approximation of the entries of the bottom-right $(n-2) \times (n-2)$ submatrix of
$ \mathbf{A_{\delta}} $ is perfect, i.e., $a_{ij} = \frac{{w}_i^{EM}}{{w}_j^{EM}} \, \, (i,j=3,4,\dots,n).$
It remains to check the approximations of the elements in the f{\kern0pt}irst and second rows and columns.
\newpage
\begin{lemma}
\label{a12}
If $\delta>1,$ then $\frac{w_1^{EM}}{w_2^{EM}} < a_{12}$.
If $\delta<1,$ then $\frac{w_1^{EM}}{w_2^{EM}} > a_{12}$.
\end{lemma}
\begin{proof}
Let $\delta>1.$
From (\ref{EQ:Farkas26}), the approximation of $a_{12}=\delta x_1$ is
\[\frac{w_1^{EM}}{w_2^{EM}}=\frac{\lambda-1+\delta}{\frac{1}{x_1}(\lambda-1+\frac{1}{\delta})}=x_1 \frac{\lambda-1+\delta}{\lambda-1+\frac{1}{\delta}}. \]
We shall prove that
$ \frac{\lambda-1+\delta}{\lambda-1+\frac{1}{\delta}} <\delta, $
that is,
$\lambda -1 + \delta < \delta \lambda - \delta +1$
$\Leftrightarrow$
$2(\delta-1) < \lambda (\delta-1)$
$\Leftrightarrow$
$2 < \lambda$. The last inequality holds, because $\lambda > n \geq 3$.
The case $\delta < 1$ is analogous.
\end{proof}

\begin{lemma}
\label{a1j}
Let $ j  > 2 $.
If $\delta>1,$ then $\frac{w_1^{EM}}{w^{EM}_{j}} > a_{1j}$.
If $\delta<1,$ then $\frac{w_1^{EM}}{w^{EM}_{j}} < a_{1j}$.
\end{lemma}
\begin{proof}
Let $\delta>1.$
From (\ref{EQ:Farkas24}), the approximation of $a_{1j}=x_{j-1}$
$(j=3,4,\dots,n)$  is
\[
\frac{w_1^{EM}}{w^{EM}_{j}}=\frac{\lambda
(\lambda-n+1)}{\frac{1}{x_{j-1}}(\lambda-1+\frac{1}{\delta})}=x_{j-1}\frac{\lambda
(\lambda-n+1)}{\lambda-1+\frac{1}{\delta}}   .
\]
The proposition is equivalent to
$\frac{\lambda \left( \lambda - n + 1\right)}{\lambda - 1 + \frac{1}{\delta} }  > 1$
$\Leftrightarrow$
$\lambda (\lambda - n + 1 ) > \lambda - 1 + \frac{1}{\delta} $
$\Leftrightarrow$
$\lambda^2 - \lambda n + \lambda > \lambda - 1 + \frac{1}{\delta}$
$\Leftrightarrow$
$(\lambda^2 - \lambda n) + \left(1 - \frac{1}{\delta}\right) > 0$. The f{\kern0pt}irst expression is positive because $\lambda >n$, and the second one is positive because $\delta >1$.\\
Now let $\delta < 1.$
From (\ref{EQ:Farkas26}), the approximation of $a_{1j}=x_{j-1}$
$(j=3,4,\dots,n)$  is  %
\[
\frac{w_1^{EM}}{w^{EM}_{j}}=
\frac{\lambda-1+\delta}{\frac{1}{x_{j-1}}\lambda \frac{\lambda-2}{n-2}} =
x_{j-1} \frac{(n-2)(\lambda-1+\delta)}{\lambda(\lambda-2)},
\]
the objective is to prove $(n-2)(\lambda-1+\delta) < \lambda(\lambda-2)$
$\Leftrightarrow$
$\lambda(n-\lambda) + (\delta-1)(n-2)  < 0$.
The f{\kern0pt}irst product is negative, because $\lambda > n$, the second product
is also negative, because $\delta < 1.$
\end{proof}

\begin{lemma}
\label{a2j}
Let $ j >2$.
If $\delta>1,$  then $\frac{w_2^{EM}}{w^{EM}_{j}} < a_{2j}$.
If $\delta<1,$  then $\frac{w_2^{EM}}{w^{EM}_{j}} > a_{2j}$.
\end{lemma}
\begin{proof}
Let $\delta>1.$
From (\ref{EQ:Farkas24}), the approximation of
$a_{2j}=\frac{x_{j-1}}{x_1}$ $(j=3,4,\dots,n)$ is
\[\frac{w_2^{EM}}{w^{EM}_{j}}=\frac{\frac{1}{x_1}\left[\lambda-(1-\frac{1}{\delta})(\lambda-n+2)\right]}{\frac{1}{x_{j-1}}(\lambda-1+\frac{1}{\delta})}=
\frac{x_{j-1}}{x_1}\frac{\lambda-(1-\frac{1}{\delta})(\lambda-n+2)}{\lambda-1+\frac{1}{\delta}} .
\]
 Thus, the proposition becomes equivalent to
$\frac{\lambda - \left( 1- \frac{1}{\delta} \right)\left( \lambda - n +2 \right)}{\lambda - \left( 1- \frac{1}{\delta} \right)} \frac{x_{j-1}}{x_1} < \frac{x_{j-1}}{x_1}$
$\Leftrightarrow$
$\lambda - \left( 1- \frac{1}{\delta} \right)\left( \lambda - n + 2 \right) < \lambda - \left( 1- \frac{1}{\delta} \right)$
$\Leftrightarrow$
$- \left( \lambda - n +2  \right) < - 1$
$\Leftrightarrow$
$\lambda - n > -1$, that holds, because $\lambda >n$.
The proof of case $\delta < 1$ is analogous to the reciprocals of the above described assertions.
\end{proof}

The next theorem states the main result of our paper:

\begin{theorem}
\label{maintheorem}  
The principal right eigenvector of a simple perturbed pairwise comparison matrix is ef{\kern0pt}f{\kern0pt}icient.
\end{theorem}
\textbf{First proof:} \\
Suppose that $\delta>1$ and apply Lemmas \ref{a12}, \ref{a1j} and \ref{a2j}.
Suppose there exists a weight vector $\mathbf{w}^{\prime}=(w_1^{\prime},w_2^{\prime},w_3^{EM},w_4^{EM},\dots,w_n^{EM})^{\top}$
that approximates matrix $\mathbf{A}_{\delta}$  at least as well as $\mathbf{w}^{EM}$ does,
and strictly better  than $\mathbf{w}^{EM}$   in one position.
 As is readily seen from    Lemma \ref{a1j}, ${w}_1^{\prime} \leq {w}_1^{EM}$.
Similarly, from Lemma \ref{a2j}, $ {w}_2^{\prime} \geq {w}_2^{EM}$.
At least one of these inequalities  must be  strict, otherwise $\mathbf{w}^{\prime} = \mathbf{w}^{EM}$.
They, together with Lemma \ref{a12}, imply that
\[
 \frac{{w}_1^{\prime}}{{w}_2^{\prime}} < \frac{{w}_1^{EM}}{{w}_2^{EM}} < a_{12},
\]
therefore  $\mathbf{w}^{\prime}$  provides a strictly worse approximation for $a_{12}$ than $\mathbf{w}^{EM}$
 does, which contradicts the initial supposition.
The case $\delta < 1$ is analogous.  \QED

\bigskip
The First proof is simple and requires no prior knowledge in multi-objective optimization problems.
Theorem 3.1. is explicitly based on matrix theory that provides a proof for the ef{\kern0pt}f{\kern0pt}iciency of the
 principal right eigenvector of  PCMs  with the specif{\kern0pt}ic structure of matrix
 $\mathbf{A}_{\delta}.$
Additionally, by depicting the digraph representation of the studied problem one can
easily visualize and check for such PCMs whether or not these solutions are, in fact,
  ef{\kern0pt}f{\kern0pt}icient. \\

\noindent
\textbf{Second proof:} \\
Let us apply Def{\kern0pt}inition \ref{def:DirectedGraphEfficient} in order to draw the
digraph associated to matrix $ \mathbf{A}_{\delta}$ and its principal right eigenvector $\mathbf{w}^{EM}$.
Namely, an arc goes from node $i$ to node $j$ if and only if $\frac{{w}_i^{EM}}{{w}_j^{EM}} \geq a_{ij}$.
Let $\delta>1$.
Lemma \ref{a12} implies that $(2 \rightarrow 1) \in \overrightarrow{E}$ and
$(1 \rightarrow 2) \notin \overrightarrow{E}.$
Lemma \ref{a1j} implies that $(1 \rightarrow j) \in \overrightarrow{E}$ and
$(j \rightarrow 1) \notin \overrightarrow{E}$ for all $j=3,4,\ldots,n$.
Lemma \ref{a2j} implies that $(2 \rightarrow j)  \notin   \overrightarrow{E}$ and
$(j \rightarrow 2)  \in  \overrightarrow{E}$ for all $j=3,4,\ldots,n$.
For $i,j=3,4,\ldots,n, \, i \neq j, \, \,$
$a_{ij} = \frac{{w}_i^{EM}}{{w}_j^{EM}}$ implies that
$(i \rightarrow j) \in \overrightarrow{E}$ and $(j \rightarrow i) \in \overrightarrow{E}.$
 The digraph is drawn in Figure 2.

\unitlength 1mm
\begin{center}
\begin{picture}(100,55)
\put(14,10){\resizebox{70mm}{!}{\rotatebox{0}{\includegraphics{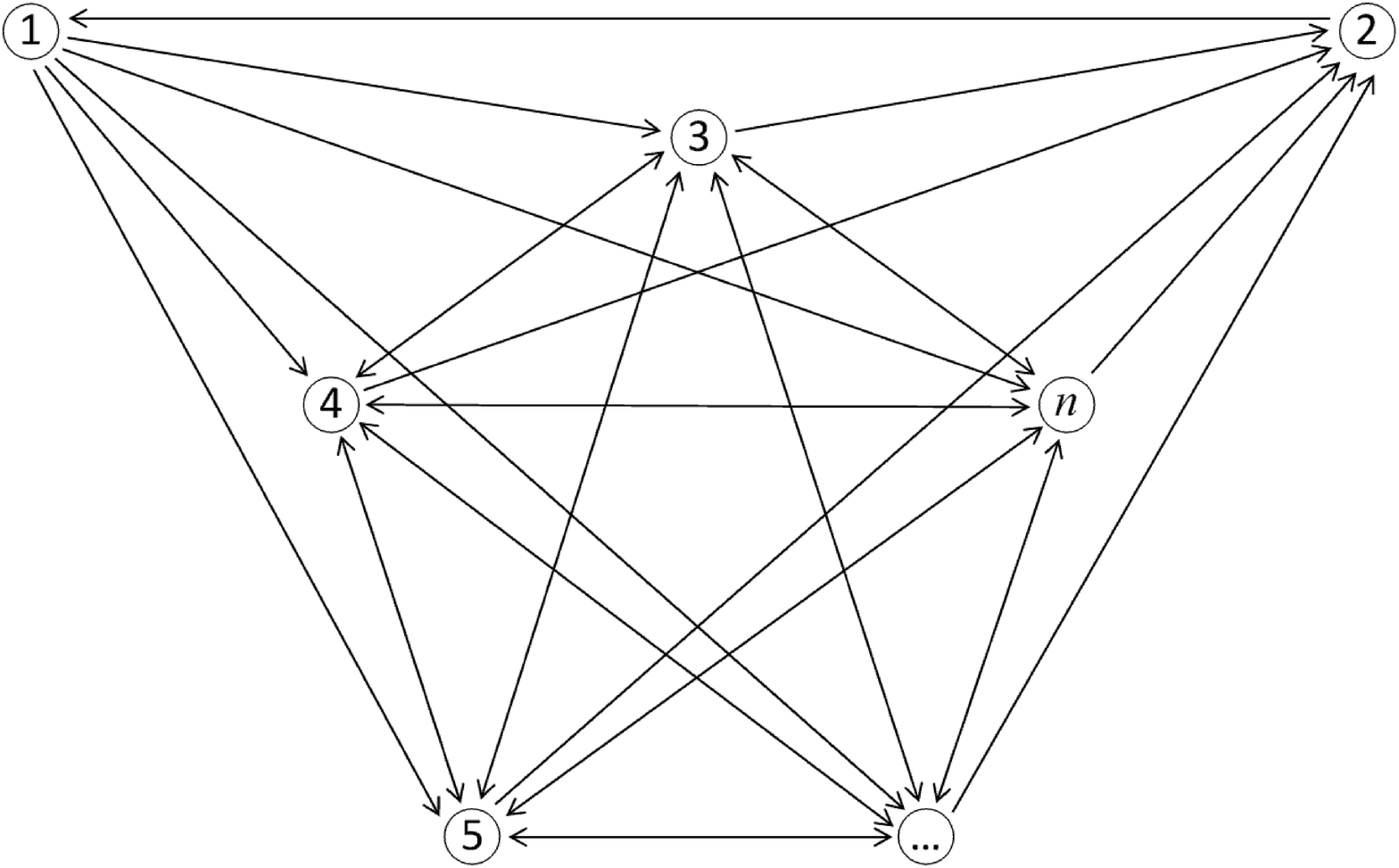}}}}
\put(-10,3){\makebox{\textbf{Figure 2.}
The strongly connected  digraph  of a simple perturbed }}
\put(18,-1){\makebox{ pairwise comparison matrix $(\delta>1)$}}
\end{picture}
\end{center}

\bigskip

The  digraph  is strongly connected,
Theorem \ref{thm:TheoremDirectedGraphEfficient} can readily be used to show that the eigenvector
$\mathbf{w}^{EM}$ is ef{\kern0pt}f{\kern0pt}icient.
The case $\delta < 1$ is similar:
the corresponding digraph is the same as that of displayed in Figure 2 except that
nodes 1 and 2 would be interchanged.  \QED

\bigskip

 Our experiments have shown that the characterization of
ef{\kern0pt}f{\kern0pt}iciency (Theorem \ref{thm:TheoremDirectedGraphEfficient}) with the
use of directed graphs (according to Def{\kern0pt}inition \ref{def:DirectedGraphEfficient})
is, indeed, very robust and this approach seems to be applicable in more complicated cases, too.

\begin{corollary}
The principal right eigenvector of a $3 \times 3$ PCM is ef{\kern0pt}f{\kern0pt}icient.
\end{corollary}
\begin{proof}
The claim follows from Remark \ref{remark:3x3} and Theorem \ref{maintheorem}.
It is worth noting that ef{\kern0pt}f{\kern0pt}iciency follows also from the equivalence of the
eigenvector method and the row geometric mean, also known as the optimal solution to the logarithmic
least squares problem \cite[Section 3.2]{Dijkstra2013}. Blanquero, Carrizosa and Conde \cite[Corollary 7]{BlanqueroCarrizosaConde2006}
proved that the weight vector calculated by the row geometric mean is ef{\kern0pt}f{\kern0pt}icient.
\end{proof}

\section{A numerical example}
\begin{example} \label{example2}  
Let us choose $n=4, x_1 = 2, x_2 = 4, x_3 = 8, \delta = 1.5$ in formula (\ref{EQ:SimplePerturbedPCM}):
\[
\mathbf{A}_{1.5} =
\begin{pmatrix}
$\,$  1   $\,\,$ & $\,\,$   3   $\,\,$ & $\,\,$   4  $\,\,$ & $\,\,$    8  $\,$    \\
$\,$ 1/3  $\,\,$ & $\,\,$   1   $\,\,$ & $\,\,$   2  $\,\,$ & $\,\,$    4  $\,$    \\
$\,$ 1/4  $\,\,$ & $\,\,$  1/2  $\,\,$ & $\,\,$   1  $\,\,$ & $\,\,$    2  $\,$   \\
$\,$ 1/8  $\,\,$ & $\,\,$  1/4  $\,\,$ & $\,\,$  1/2 $\,\,$ & $\,\,$    1   $\,$
\end{pmatrix}.
\]
Matrix $\mathbf{A}_{1.5}$ is a simple perturbed PCM.
Its principal right eigenvector $\mathbf{w}^{EM}$ and the consistent
approximation of $\mathbf{A}_{1.5}$, generated by $\mathbf{w}^{EM}$,
are displayed, truncated at 8 and 4 correct digits, respectively:
\[
\mathbf{w}^{EM} =
\begin{pmatrix}
  0.57313428  \\
  0.23374121  \\
  0.12874966  \\
  0.06437483
\end{pmatrix},
\quad
\left[ \frac{w^{EM}_i}{w^{EM}_j} \right] =
\begin{pmatrix}
$\,$   1    $\,\,$ & $\,\,$ 2.4520 $\,\,$ & $\,\,$ 4.4515 $\,\,$ & $\,\,$ 8.9030 $\,$    \\
$\,$ 0.4078 $\,\,$ & $\,\,$   1    $\,\,$ & $\,\,$ 1.8154 $\,\,$ & $\,\,$ 3.6309 $\,$    \\
$\,$ 0.2246 $\,\,$ & $\,\,$ 0.5508 $\,\,$ & $\,\,$    1   $\,\,$ & $\,\,$   2    $\,$   \\
$\,$ 0.1123 $\,\,$ & $\,\,$ 0.2754 $\,\,$ & $\,\,$   1/2  $\,\,$ & $\,\,$   1    $\,$
\end{pmatrix}.
\]
One can verify that entry 3 in position (1,2) of $\mathbf{A}_{1.5}$
is underestimated by $\frac{w^{EM}_1}{w^{EM}_2} = 2.4520$, in accordance with Lemma (\ref{a12})
and represented by the arc from node 2 to node 1 in Figure 2. Lemmas (\ref{a1j})-(\ref{a2j}) can
also be checked. Finally, entry 2 in position (3,4) of $\mathbf{A}_{1.5}$
is estimated perfectly by $\frac{w^{EM}_3}{w^{EM}_4} = 2$, represented by a
bi-directed edge between nodes 3 and 4. The directed graph in Figure 2 is
strongly connected, which ensures that the principal right eigenvector
is ef{\kern0pt}f{\kern0pt}icient.
\end{example}  

\unitlength 1mm
\begin{center}
\begin{picture}(100,40)
\put(39,10){\resizebox{20mm}{!}{\rotatebox{0}{\includegraphics{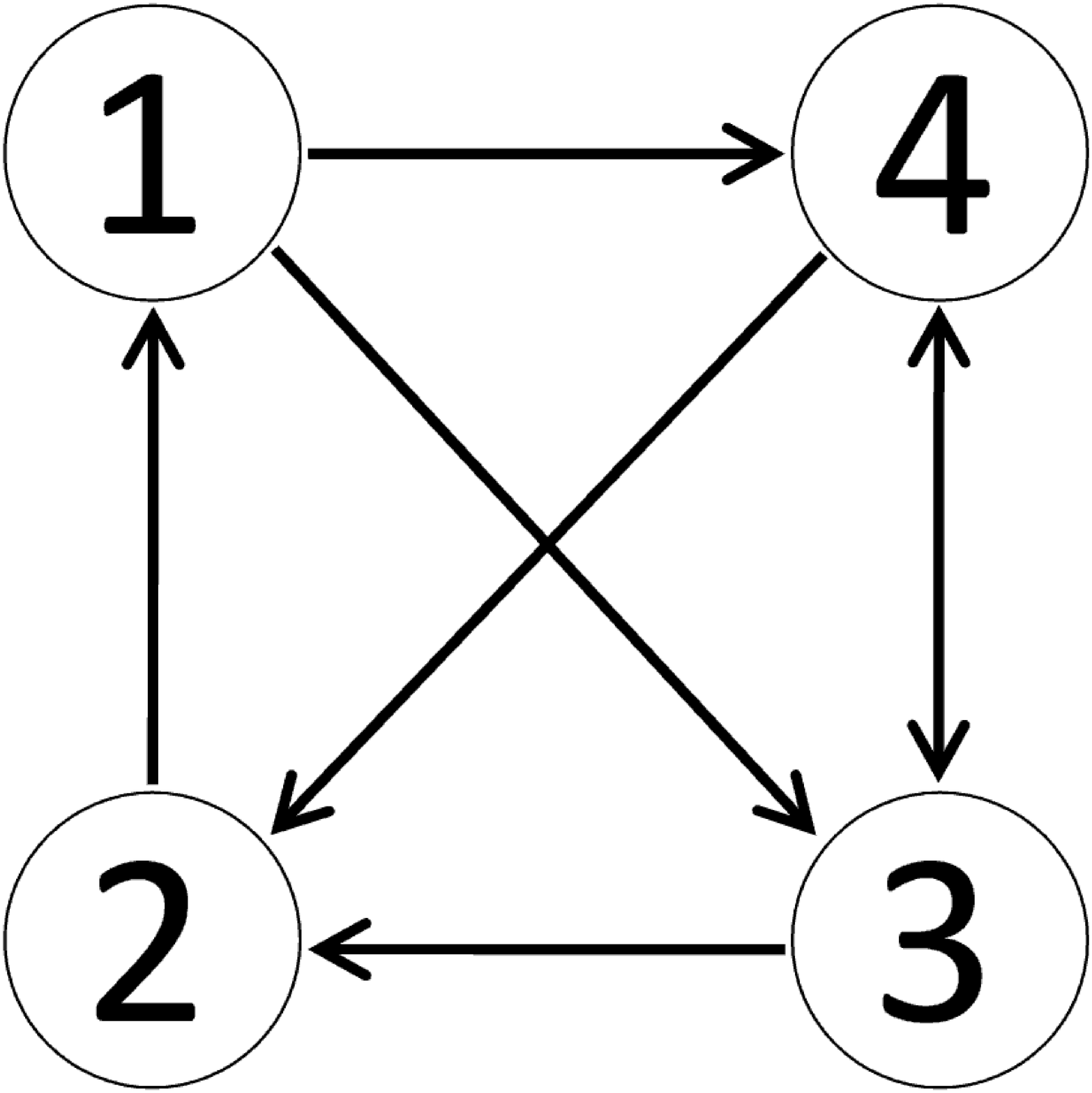}}}}
\put(-13,3){\makebox{\textbf{Figure 2.}
The principal right eigenvector in Example 5.1 is ef{\kern0pt}f{\kern0pt}icient, }}
\put(5,-1){\makebox{because its associated digraph is strongly connected }}
\end{picture}
\end{center}

\section{Conclusions and open questions}
We have made a small step in the  presumably  long way of
f{\kern0pt}inding a necessary and suf{\kern0pt}f{\kern0pt}icient condition of that
the principal right eigenvector be ef{\kern0pt}f{\kern0pt}icient. It has been shown
in the paper that if the PCM is
simple perturbed, that is, it can be made consistent by
a modif{\kern0pt}ication of an element and its reciprocal, then
the principal right eigenvector is ef{\kern0pt}f{\kern0pt}icient.
The explicit formulas of the principal right eigenvector
of a simple perturbed PCM enabled us to provide a  solid  proof.
We hope to return to a similar special case, when
the PCM can be made consistent by a modif{\kern0pt}ication of two elements
and their reciprocals.
However, in the general case, explicit formulas do not exist,
or, even if they exist, they may be hopelessly complicated.

Although we would not overrate the practical signif{\kern0pt}icance of simple
perturbed PCMs, they do occur in real decision problems \cite[Table 1]{BozokiFulopPoesz2011}.
 On the other hand, they might help understanding the phenomenon of
(in)ef{\kern0pt}f{\kern0pt}iciency, which is enigmatic at the moment.
The question of the possible relations between the ef{\kern0pt}f{\kern0pt}iciency of the principal
right eigenvector and the level of inconsistency is also
to be investigated.

The authors believe that ef{\kern0pt}f{\kern0pt}iciency is a reasonable and desirable property,
independently of the background of the analyst. Being an economist,
engineer, decision theorist, preference modeler or mathematician,
who faces an estimation problem, inef{\kern0pt}f{\kern0pt}icient solutions are hard to argue for.
 The comparative studies of weighting methods,
such as \cite{BajwaChooWedley2008,ChooWedley2004,Dijkstra2013,GolanyKress1993},
should be extended by adding ef{\kern0pt}f{\kern0pt}iciency to the list of axioms/criteria.

\section{Acknowledgments}
The authors would like to thank the three anonymous reviewers for their valuable
and constructive recommendations.
The authors are grateful to J\'ozsef Temesi
(Department of Operations Research and Actuarial Sciences, Corvinus University of Budapest)
for raising the idea of the ef{\kern0pt}f{\kern0pt}iciency analysis of simple perturbed pairwise comparison matrices.
 Comments of J\'anos F\"ul\"op
(Institute for Computer Science and Control, Hungarian Academy of Sciences (MTA SZTAKI) and \'Obuda University, Budapest)
and \"Ors Reb\'ak (Corvinus University of Budapest)
are highly appreciated.
Research was supported in part by OTKA grant K 111797.\\

This paper is dedicated to the second author's grandmother, who celebrated her 100th birthday on September 15, 2015.

\bibliographystyle{plain}
\bibliography{Abele-NagyBozoki}

\end{document}